\documentclass[11pt]{article}

\usepackage[T1]{fontenc}
\usepackage[utf8]{inputenc}
\usepackage{lmodern}
\usepackage{microtype}
\usepackage[a4paper,left=23mm,right=23mm,top=23mm,bottom=23mm]{geometry}
\usepackage{setspace}
\usepackage{amsmath,amssymb,amsthm,mathtools}
\usepackage{enumitem}
\usepackage{xcolor}
\usepackage[colorlinks=true,linkcolor=blue!55!black,citecolor=blue!55!black,urlcolor=blue!55!black]{hyperref}
\usepackage[nameinlink,capitalise,noabbrev]{cleveref}

\setlist[itemize]{leftmargin=2.1em,itemsep=0.32em,topsep=0.42em}
\setlist[enumerate]{leftmargin=2.3em,itemsep=0.32em,topsep=0.42em}
\allowdisplaybreaks

\newtheorem{theorem}{Theorem}[section]
\newtheorem{lemma}[theorem]{Lemma}
\newtheorem{proposition}[theorem]{Proposition}

\newtheorem{conjecture}[theorem]{Conjecture}

\theoremstyle{definition}

\newtheorem{remark}[theorem]{Remark}

\crefname{theorem}{Theorem}{Theorems}
\crefname{lemma}{Lemma}{Lemmas}
\crefname{proposition}{Proposition}{Propositions}
\crefname{corollary}{Corollary}{Corollaries}
\crefname{definition}{Definition}{Definitions}
\crefname{remark}{Remark}{Remarks}
\crefname{conjecture}{Conjecture}{Conjectures}

\theoremstyle{plain}
\newtheorem*{thmA}{Theorem A}
\newtheorem*{conjB}{Conjecture B}
\newtheorem*{thmC}{Theorem C}
\newtheorem*{conjD}{Conjecture D}

\newcommand{\EE}{\mathbb E}
\newcommand{\PP}{\mathbb P}

\newcommand{\cF}{\mathcal F}
\newcommand{\cG}{\mathcal G}
\newcommand{\cL}{\mathcal L}

\newcommand{\angles}[1]{\langle #1\rangle}
\newcommand{\Aut}{\operatorname{Aut}}
\newcommand{\Inj}{\operatorname{Inj}}
\newcommand{\Autr}{\operatorname{Aut}^{\circ}}
\newcommand{\Injr}{\operatorname{Inj}^{\circ}}
\newcommand{\e}{\mathrm e}
\newcommand{\pe}{p_{\mathbb E}}
\newcommand{\pef}{p_{\mathbb E}^{*}}
\newcommand{\pc}{p_{\mathrm c}}
\newcommand{\secref}[1]{\ref*{#1}}

\title{Fractional expectation thresholds\\ and the ``second'' Kahn--Kalai conjecture}
\author{Tuan Tran
  \thanks{School of Mathematical Sciences, University of Science and Technology of China. Supported by the Excellent Young Talents Program (Overseas) of the National Natural Science Foundation of China under Grant No. GG0010007003. \texttt{trantuan@ustc.edu.cn}}
  }
\date{}

\hypersetup{pdftitle={Fractional expectation thresholds and the "second" Kahn--Kalai conjecture},
pdfauthor={Tuan Tran}}

\begin{document}
\maketitle

\begin{abstract}
We show that the uniform probability measure on copies of a
nonempty graph $H$ in $K_n$ is $Cq_H\log(2e(H))$-spread,
where $q_H$ is its graphic expectation threshold. Consequently, the fractional expectation threshold of $H$ is at most $Cq_H\log(2e(H))$.
We remove the logarithmic factor for trees and for graphs
whose average degree is at least the logarithm of their
maximum degree. This proves the ``second'' Kahn--Kalai
conjecture for these two classes, which encompass most
of the standard families studied in random graph
containment problems.
\end{abstract}

\section{Introduction}\label{sec:introduction}

Determining the threshold for a random graph to contain a prescribed
graph is  a central problem in random graph theory.
This problem encompasses both the appearance of fixed subgraphs
and the emergence of spanning structures. A basic question is how closely the threshold
is predicted by first-moment estimates for the subgraphs of the
target graph.

We first recall the general framework of expectation thresholds.  Let $\cF$ be a nonempty
increasing family of subsets of a finite set $X$, with $\varnothing\notin\cF$,
and let $X_p$ include each element of $X$ independently with probability $p$.
Its critical probability is
\[
\pc(\cF):=\inf\{p:\PP(X_p\in\cF)\geq 1/2\}.
\]
A cover of $\cF$ is a family $\cG$ such that every member of $\cF$ contains
some $S\in\cG$.  Its cost at $p$ is $\sum_{S\in\cG}p^{|S|}$.  The union bound
shows that this cost is an upper bound on $\PP(X_p\in\cF)$, so the
\emph{expectation threshold}
\[
q(\cF):=\sup\{p:\text{some cover of $\cF$ has cost at most $1/2$ at $p$}\}
\]
satisfies $q(\cF)\le\pc(\cF)$.  Park and Pham famously proved the conjecture of Kahn
and Kalai \cite{KK} that these quantities differ by at most a logarithmic factor.

\begin{thmA}[Park and Pham \cite{PP}]
There is an absolute constant $C$ such that every such increasing family
$\cF$ satisfies
\[
\pc(\cF)\le Cq(\cF)\log(2\ell(\cF)),
\]
where $\ell(\cF)$ is the largest size of a minimal member of $\cF$.
\end{thmA}

The main difficulty in applying this theorem is often to estimate $q(\cF)$:
an upper bound must rule out every cover with small cost.  For graph
containment, we can ask for a bound based directly on subgraph counts.
Fix a graph $H$ with at least one edge and at most $n$ vertices, and let
$\cF_H$ be the family of graphs on $[n]$ containing a copy of $H$.  Write
$\pc(H):=\pc(\cF_H)$.  Throughout, containment is non-induced, and isolated
vertices are omitted unless stated otherwise.

For graphs $I$ and $G$, let $\Inj(I,G)$ count the injective maps from $V(I)$
to $V(G)$ that preserve edges.  Put $\Aut(I):=\Inj(I,I)$ and
$N_I(G):=\Inj(I,G)/\Aut(I)$, the number of copies of $I$ in $G$.  If
$X_I:=N_I(G_{n,p})$, then $\EE_pX_I=N_I(K_n)p^{e(I)}$.  The
\emph{graphic expectation threshold} is
\[
\pe(H):=\inf\{p:\EE_pX_I\geq 1/2
\text{ for every nonempty }I\subseteq H\}.
\]
For every such $I$, its copies in $K_n$ cover
$\cF_H$, so $\pe(H)\le q(\cF_H)\le\pc(H)$.  The following conjecture asks
whether Theorem A remains true with $\pe(H)$ in place of $q(\cF_H)$.

\begin{conjB}[Kahn and Kalai \cite{KK}]
There is an absolute constant $C$ such that every nonempty $H\subseteq K_n$
satisfies
\[
\pc(H)\le C\pe(H)\log(2e(H)).
\]
\end{conjB}

Although now called the ``second'' Kahn--Kalai conjecture
\cite{MNSZ,DKP}, this was the starting point for Kahn and Kalai's original
work \cite{KK,DKP}.

Fractional covers provide another way to bound the threshold.
Allowing nonnegative weights in the definition of a cover yields
Talagrand's fractional expectation threshold $q_f(\cF)$, which
satisfies $q(\cF)\le q_f(\cF)\le\pc(\cF)$.
Frankston, Kahn, Narayanan, and Park proved the following
fractional version of the Kahn--Kalai conjecture, proposed
by Talagrand \cite{Talagrand10}.

\begin{thmC}[Frankston, Kahn, Narayanan, and Park \cite{FKNP}]
There is an absolute constant $C$ such that every nonempty increasing family
$\cF$ with $\varnothing\notin\cF$ satisfies
\[
\pc(\cF)\le Cq_f(\cF)\log(2\ell(\cF)),
\]
where $\ell(\cF)$ is the largest size of a minimal member of $\cF$.
\end{thmC}

Write $\pef(H):=q_f(\cF_H)$.  Since the minimal members of $\cF_H$ are the
copies of $H$, Theorem C gives
$\pc(H)\le C\pef(H)\log(2e(H))$.  Thus Conjecture B would follow from
$\pef(H)\le C\pe(H)$.

To bound $\pef(H)$, we use spread probability measures.
A probability measure $\mu$ on subgraphs of $K_n$ is
\emph{$a$-spread} if
$\mu(\{S:A\subseteq E(S)\})\le a^{|A|}$
for every nonempty edge set $A\subseteq E(K_n)$.
An $a$-spread measure supported on copies of $H$ gives
$\pef(H)\le a$, and hence a threshold bound by Theorem~C.

We focus on the uniform probability measure $\nu_H$ on
the copies of $H$ in $K_n$, for which spreadness can be
expressed in terms of copy counts. Indeed, if
$A\subseteq E(K_n)$ is nonempty and $J$ is the graph
formed by its edges and their endpoints, then
\[
\nu_H(\{S:A\subseteq E(S)\})
=\frac{N_J(H)}{N_J(K_n)}
=\frac{\Inj(J,H)}{(n)_{v(J)}}.
\]
This identity, used by Mossel, Niles-Weed, Sun, and Zadik
\cite{MNSZ}, is recalled in \cref{lem:cylinder-identity}.

We replace the expected count $1/2$ by $1$ and set
\[
q_H:=\max_{\varnothing\neq I\subseteq H}N_I(K_n)^{-1/e(I)}.
\]
Thus $\EE_{q_H}X_I\geq1$ for every nonempty $I\subseteq H$, and
$\pe(H)\le q_H\le2\pe(H)$.  Dubroff, Kahn, and Park proposed the following
counting conjecture, which we state here as an equivalent spread bound.

\begin{conjD}[Dubroff, Kahn, and Park {\cite{DKP}}]
There is an absolute constant $C$ such that $\nu_H$ is $Cq_H$-spread for every
nonempty $H\subseteq K_n$.
\end{conjD}

The fractional Kahn--Kalai theorem (Theorem~C), together
with $q_H\le2\pe(H)$, shows that Conjecture~D implies
Conjecture~B. In a breakthrough toward Conjecture~D,
Dubroff, Kahn, and Park \cite{DKP} proved that
$\nu_H$ is $Cq_H\log^2 n$-spread for every nonempty
$H\subseteq K_n$. To put this advance in perspective,
Keevash, Lifshitz, Long, and Minzer \cite{KLLM} had noted
that Conjecture~B was ``widely open'' even with the
logarithmic factor replaced by $n^{o(1)}$.

\subsection{Our results}

We prove the second Kahn--Kalai conjecture for trees and for graphs
whose average degree is at least the logarithm of their maximum
degree. For arbitrary graphs, we obtain a bound with one additional
logarithm. Our approach is based on counting copies of subgraphs inside the target graph.

\begin{theorem}\label{thm:one-log-fractional}
There is an absolute constant $C$ such that $\nu_H$ is
$Cq_H\log(2e(H))$-spread for every nonempty $H\subseteq K_n$.
\end{theorem}

Consequently, $\pef(H)\le C\pe(H)\log(2e(H))$, and Theorem C
gives $\pc(H)\le C\pe(H)\log^2(2e(H))$.
Thus Conjecture B holds up to one additional logarithmic factor.

Our next result relates the loss to the degree and density of $H$.  Write
$\Delta(H)$ for its maximum degree and
$d_*(H):=\max_{\varnothing\neq J\subseteq H}\frac{2e(J)}{v(J)}$
for its maximum subgraph average degree.  
\begin{theorem}\label{thm:density}
There is an absolute constant $C$ such that $\nu_H$ is
$C\Delta(H)^{2/d_*(H)}q_H$-spread for every nonempty $H\subseteq K_n$.
\end{theorem}

The factor
$\Delta(H)^{2/d_*(H)}$ is bounded whenever $d_*(H)=\Omega(\log\Delta(H))$.
In particular, this factor is bounded for graphs with a bounded ratio of
maximum to minimum degree, with a constant depending only on that ratio. Thus Conjectures D and B hold
for all nonempty regular graphs, with constants independent
of the degree.

Trees can have large maximum degree while $d_*(H)<2$.
Our third result removes the degree dependence for this class.

\begin{theorem}\label{thm:count}
There is an absolute constant $C$ such that $\nu_H$ is
$Cq_H$-spread for every tree $H\subseteq K_n$ with at least
one edge.
\end{theorem}

This proves Conjecture D for trees. It also gives the bounds
$\pef(H)\le C\pe(H)$ and
$\pc(H)\le C\pe(H)\log(2e(H))$, which imply Conjecture B
for trees.
In their companion paper, Dubroff, Kahn, and Park \cite{DKP2}
proved the counting bound in Conjecture D for clique, cycle,
and bounded-degree forest witnesses inside an arbitrary $H$.
They explicitly asked to remove the degree restriction on tree
witnesses. \Cref{thm:count} gives this extension when the host
$H$ is itself a tree. It does not settle their question for
arbitrary hosts.

For comparison, Montgomery \cite{Montgomery} proved that, for each fixed
$\Delta$, the random graph $G_{n,C_\Delta\log n/n}$ contains every $n$-vertex
tree of maximum degree at most $\Delta$ simultaneously with high probability.
For bounded-degree spanning trees, $q_H=O_\Delta(1/n)$, so \cref{thm:count}
recovers the $\log n/n$ order for the threshold of each individual tree.
It also gives $\pc(H)\le Cq_H\log n$ without a degree restriction.  This
extends the threshold bound for individual trees to the scale given by
$q_H$. It does not show that all such trees appear at once.

\subsection{Methodology}
Dubroff, Kahn, and Park \cite{DKP} prove their general
bound by decomposing a fixed witness into rooted pieces and counting
its extensions.  For \cref{thm:one-log-fractional}, we construct a single
chain of subgraphs of $H$ that controls every witness.  At each step,
we maximize the ratio of a subgraph's containment probability under the
current uniform copy law to a proposed $Cq_H$-spread bound.  The
first-moment constraints force the resulting core to have at most a
fixed fraction of the current edges, while maximality makes the
remaining edges $Cq_H$-spread conditional on any copy of the core.

Iterating gives $L=O(\log(2e(H)))$ blocks.  Starting with a uniform copy
of the last core, we extend uniformly along the chain; symmetry
preserves the uniform copy law at each step.  Conditional spreadness bounds
the probability of each assignment of $t$ fixed witness edges to the
blocks by $(Cq_H)^t$.  Summing over
the $L^t$ assignments gives the single logarithmic loss in
\cref{thm:one-log-fractional}.  Under the coupling conjecture in
\cref{sec:coupling-conjecture}, we could absorb these blocks successively
into a product measure of density $O(q_H\log(2e(H)))$, proving
Conjecture B.

The other two proofs start from a simple packing observation of Dubroff,
Kahn and Park \cite{DKP}:  If $J$ is connected and $b_H(J)$ is the
maximum number of vertex-disjoint copies of $J$ in $H$, then
\[
b_H(J)\le \frac{\e n^{v(J)}q_H^{e(J)}}{\Aut(J)}.
\]
To pass from disjoint copies to all copies,
\cref{sec:density} uses the vertices of a maximal packing to start a
spanning-tree embedding count.  In \cref{sec:trees}, we collect tree
embeddings in a rooted subtree sum, compare it with its largest weighted
term, and use packings to average these largest terms over $H$.

\section{Preliminaries}\label{sec:preliminaries}

All graphs are finite and simple; a nonempty graph has at least one
edge. Subgraphs need not be connected, and isolated vertices are
omitted unless explicitly allowed. All logarithms are natural, and
absolute constants may change from line to line. We use the falling
factorial $(n)_v:=n(n-1)\cdots(n-v+1)$, with $(n)_0=1$.
The definitions of $\pe(H)$ and $q_H$ give
\begin{equation}\label{eq:q-versus-pe}
\pe(H)\le q_H\le2\pe(H).
\end{equation}

For $F\subseteq K_n$, let $\nu_F$ be the uniform probability
measure on the copies of $F$ in $K_n$. For $A\subseteq E(K_n)$,
write $\angles A:=\{S\subseteq K_n:A\subseteq E(S)\}$.
We recall the standard counting identity \cite[eq.~(5)]{MNSZ}.

\begin{lemma}\label[lemma]{lem:cylinder-identity}
Let $A$ be a nonempty edge set and let $J$ be the graph formed by
its edges and their endpoints. Then
\[
\nu_F(\angles A)
=\frac{N_J(F)}{N_J(K_n)}
=\frac{\Inj(J,F)}{(n)_{v(J)}}.
\]
\end{lemma}

A fractional cover of the copies of $H$ in $K_n$ assigns weights
$w_A\geq0$ to $A\subseteq E(K_n)$ such that
$\sum_{A\subseteq E(S)}w_A\geq1$ for every copy $S$ of $H$ in $K_n$.
Its cost at $p$ is $\sum_Aw_Ap^{|A|}$, and $\pef(H)$ is the
supremum of the $p$ for which a cover of cost at most $1/2$ exists.

The following consequence of fractional-cover duality and the
fractional Kahn--Kalai theorem is standard; see
\cite[Observation~7.4 and Theorem~1.1]{FKNP}.

\begin{lemma}\label[lemma]{lem:spread-threshold}
There is an absolute constant $C$ such that, whenever $H$ is
nonempty and $\nu_H$ is $a$-spread,
\[
\pef(H)\le a,
\qquad
\pc(H)\le Ca\log(2e(H)).
\]
\end{lemma}

Thus our task is to bound the embedding ratios in
\cref{lem:cylinder-identity}.

\section{Ratio-maximizing cores and the proof of Theorem~\secref{thm:one-log-fractional}}\label{sec:cores}

We will construct a uniform copy of $H$ as the union of a short sequence
of conditionally spread edge sets.  We first choose a subgraph whose containment
probability has the largest ratio to a proposed spread bound.  Maximizing this ratio has two consequences: the selected subgraph
has few edges, and conditioning on a copy of it makes the remaining edges
spread.  Repeating this choice reduces the edge count by a fixed factor at each
step, so we need only $O(\log(2e(H)))$ steps.

Put $a:=64\e\,q_H$. If $a\ge1$, every probability measure is
$a$-spread, and the theorem follows by increasing its absolute
constant. We therefore assume $a<1$. Throughout this section and the next, we suppress $E(\cdot)$ when treating graphs as edge sets, while retaining $v(G)$
and $e(G)$ for vertex and edge counts.

Let $F\subseteq H$ be nonempty. For every edge set
$A\subseteq E(K_n)$, define $\Psi_F(A):=\nu_F(\angles A)/a^{|A|}$.
Fix a nonempty subgraph $C(F)\subseteq F$ whose edge set
maximizes $\Psi_F$ over all such $A$, and call it the
\emph{ratio-maximizing core} of $F$.
Such a choice is possible because the maximum is positive:
every maximizer is therefore contained in a copy of $F$ and can be relabeled into $F$ without changing its value
under $\Psi_F$.
This maximizing choice gives the following conditional spread bound.

\begin{lemma}\label[lemma]{lem:core}
Let $F\subseteq H$ with $\nu_F$ not $a$-spread. Then
$e(C(F))<e(F)/64$. Moreover, let $T$ be any fixed copy
of $C(F)$ in $K_n$, and let $S$ be a uniformly random
copy of $F$ in $K_n$. Conditional on $T\subseteq S$,
the remaining edge set $S\setminus T$ is $a$-spread:
for every $U\subseteq E(K_n)$,
\[
\PP\bigl(U\subseteq S\setminus T\mid T\subseteq S\bigr)
\le a^{|U|}.
\]
\end{lemma}

\begin{proof}
Write $r=e(F)$ and $t=e(C(F))$. Since $\nu_F$ is not
$a$-spread, the maximizing choice of $C(F)$ gives
$\Psi_F(C(F))>1$. Also,
$N_{C(F)}(F)\le\binom rt$, and, since $C(F)\subseteq H$,
the definition of $q_H$ gives
$N_{C(F)}(K_n)^{-1}\le q_H^t$.
Thus \cref{lem:cylinder-identity} yields
\[
1<\Psi_F(C(F))
\le\binom rt \left(\frac{q_H}{a}\right)^t
\le\left(\frac{r}{64t}\right)^t,
\]
where the last inequality uses $\binom rt\le(\e r/t)^t$ and $a=64\e\,q_H$.
Hence $t<r/64$.

For the conditional assertion, fix a copy $T$ of $C(F)$
in $K_n$. By invariance under relabeling, $T$ also maximizes
$\Psi_F$, and $\nu_F(\angles T)>0$.
Let $U\subseteq E(K_n)$. If $U$ meets $T$, the event
$U\subseteq S\setminus T$ is impossible. Otherwise,
\[
\PP\bigl(U\subseteq S\setminus T\mid T\subseteq S\bigr)
=\frac{\nu_F(\angles{T\cup U})}{\nu_F(\angles T)}
=a^{|U|}\frac{\Psi_F(T\cup U)}{\Psi_F(T)}
\le a^{|U|},
\]
where the final inequality follows from maximality.
\end{proof}

Starting with $F_0=H$, define a descending chain of graphs.  If $\nu_{F_i}$ is
$a$-spread, stop.  Otherwise set $F_{i+1}:=C(F_i)$. By
\cref{lem:core},
\begin{equation}\label{eq:rank-geometric}
e(F_{i+1})<e(F_i)/64.
\end{equation}
Every $F_i$ is a nonempty subgraph of $H$, so the edge counts are positive integers
and the process stops.  Let $F_{\ell}$ be the last graph, where $\ell\geq0$.  If
$\ell\geq1$, then iterating \eqref{eq:rank-geometric} gives $1\le e(F_{\ell})<e(H)/64^{\ell}$,
so in every case
\begin{equation}\label{eq:chain-length}
\ell\le\log_{64}e(H).
\end{equation}

We now sample a sequence of nested copies of these graphs.  First choose $S_{\ell}$
uniformly among the copies of $F_{\ell}$.  For $i=\ell-1,\ell-2,\ldots,0$, conditionally on
$S_{i+1}$, choose $S_i$ uniformly among the copies of $F_i$ that contain
$S_{i+1}$.

\begin{lemma}\label[lemma]{lem:flag-law}
For $i=0,\ldots,\ell$, the distribution of $S_i$ is $\nu_{F_i}$.
For $i<\ell$, conditionally on $S_{i+1}$, the block $B_i:=S_i\setminus S_{i+1}$ is $a$-spread in the
sense of \cref{lem:core}.  The last block $B_{\ell}:=S_{\ell}$ has the $a$-spread
distribution $\nu_{F_{\ell}}$.
\end{lemma}

\begin{proof}
The conditional spread assertion follows directly from \cref{lem:core}, and the
last assertion holds because the chain stopped at $F_{\ell}$.  We verify the
distribution.

Assume by induction that $S_{i+1}$ is uniform.  Write $F=F_i$ and
$F'=F_{i+1}$, and let $M$ be the number of copies
of $F$ containing a fixed copy of $F'$.  Counting pairs of nested copies in two ways gives
$N_{F'}(K_n)M=N_F(K_n)N_{F'}(F)$.  For a fixed copy $Q$ of $F$ there are exactly
$N_{F'}(F)$ copies of $F'$ inside $Q$.  Therefore
\[
\PP(S_i=Q)
=
\frac{N_{F'}(F)}{N_{F'}(K_n)M}
=
\frac1{N_F(K_n)}.
\]
Thus $S_i$ is uniform.  Starting from $S_{\ell}$ and working backward proves the claim at every
step.
\end{proof}

The blocks $B_0,\ldots,B_{\ell}$ are pairwise disjoint and partition $S_0$.

\begin{lemma}\label[lemma]{lem:tagged}
Let $A_0,\ldots,A_{\ell}\subseteq E(K_n)$ be pairwise disjoint fixed edge sets.  Then
\[
\PP\bigl(A_i\subseteq B_i\text{ for }i=0,\ldots,\ell\bigr)
\le
a^{\sum_{i=0}^{\ell}|A_i|}.
\]
\end{lemma}

\begin{proof}
For $i=0,\ldots,\ell$, let $\cG_i$ be the $\sigma$-algebra
generated by $S_i,\ldots,S_\ell$, representing the information
revealed by these copies. For $0\le i<\ell$, the copies $S_{i+1},\ldots,S_\ell$ determine all the blocks $B_j$ with $j>i$, so the event $A_j\subseteq B_j$ for every $j>i$ is
$\cG_{i+1}$-measurable.  Given $\cG_{i+1}$, the conditional law of $S_i$ is uniform
on the copies of $F_i$ containing $S_{i+1}$. If $A_i$ meets $S_{i+1}$, then $A_i\subseteq B_i$ is impossible.
Otherwise \cref{lem:core} gives $\PP(A_i\subseteq B_i\mid\cG_{i+1})\le
a^{|A_i|}$.  At the last step, $\PP(A_{\ell}\subseteq B_{\ell})\le a^{|A_{\ell}|}$
because $\nu_{F_{\ell}}$ is $a$-spread.  Applying these conditional bounds in order for
$i=0,1,\ldots,\ell-1$, and then using the bound for the last block, proves the claim.
\end{proof}

\begin{proof}[Proof of \cref{thm:one-log-fractional}]
Let $A\subseteq E(K_n)$ be nonempty.  Sample the sequence of nested copies
$S_0\supset S_1\supset\cdots\supset S_{\ell}$ constructed above.  By
\cref{lem:flag-law}, $S_0$ is a uniform copy of $H$, so
$\nu_H(\angles A)=\PP(A\subseteq S_0)$.

If $A\subseteq S_0$, every edge of $A$ lies in exactly one block,
so there is a unique assignment $f:A\to\{0,1,\ldots,\ell\}$ with $f(e)=i$ if and
only if $e\in B_i$.  Conversely, for such an $f$ put
$A_i(f):=\{e\in A:f(e)=i\}$.  The events
$\{A_i(f)\subseteq B_i\text{ for }i=0,\ldots,\ell\}$ are pairwise disjoint over the
choices of $f$, and their union is exactly $\{A\subseteq S_0\}$.  By
\cref{lem:tagged} each has probability at most $a^{|A|}$, and there are $(\ell+1)^{|A|}$
assignments, so $\nu_H(\angles A)\le[(\ell+1)a]^{|A|}$.
Finally, \eqref{eq:chain-length} gives
$(\ell+1)a \le Cq_H\log(2e(H))$.
\end{proof}

\section{A coupling conjecture}\label{sec:coupling-conjecture}

The decomposition in \cref{sec:cores} also suggests a route
to Conjecture~B. Rather than combine the conditionally spread
blocks into a spread bound for $\nu_H$ and then apply
Theorem~C, we seek to absorb them successively into a product
measure. The coupling conjecture below would make this
possible at total density $O(q_H\log(2e(H)))$. 

For a finite set $X$, write $X_p$ for its $p$-random
subset, obtained by including each element independently with probability
$p$, and write $\mu_p$ for the law of $X_p$.  For laws $P,Q$ on
$2^X$, write $P\preceq_\delta Q$ if they admit a
coupling $(Y,W)$ with $\PP(Y\subseteq W)\geq1-\delta$.

\begin{conjecture}\label[conjecture]{conj:reserve-coupling}
There is an absolute constant $C$ such that, whenever $B$
is an $a$-spread random subset of $X$, $|B|\le r$ almost
surely with $r\ge1$, and $X_p$ is independent of $B$,
\[
\cL(B\cup X_p)\preceq_{1/4}\mu_{\min\{1,p+Ca\}}
\qquad\text{if }p\geq Ca\log(2r).
\]
\end{conjecture}

The scale of the increment is suggested by the marginal bound
$\PP(x\in B\cup X_p)=p+(1-p)\PP(x\in B)\le p+a$.
The logarithmic reserve is suggested by taking $B$ to choose one point
independently and uniformly from each of $r$ disjoint $m$-element
blocks, so $a=1/m$: a product set containing $B$ must hit every block,
which in the sparse regime costs density of order $a\log(2r)$.
The conjecture predicts that this logarithmic cost can be paid once
in the independent noise, after which absorbing the prescribed law of
$B$ costs only $O(a)$.  The coupling may rearrange complete noise
configurations.

\begin{remark}\label[remark]{rem:coupling-amplification}
Up to changing the absolute constant, \cref{conj:reserve-coupling} is
equivalent to the following stronger statement: under the
same hypotheses on $B$ and $X_p$, for every $0<\delta<1$,
\[
\cL(B\cup X_p)\preceq_\delta\mu_{\min\{1,p+Ca\}}
\qquad\text{if }p\geq Ca\log(2r/\delta).
\]
Indeed, let $C_\ast$ be the constant in the conjecture and put
$k=\lceil4\delta^{-2}\rceil$.  The union of $k$ independent copies of
$B$ on disjoint ground sets is $a$-spread and has size at most $kr$.
Thus, when $p\geq C_\ast a\log(2kr)$, the conjecture gives
$P^{\otimes k}\preceq_{1/4}Q^{\otimes k}$, where
$P=\cL(B\cup X_p)$ and $Q=\mu_{\min\{1,p+C_\ast a\}}$.
If an increasing family $\mathcal A$ had
$u-v>\delta$, where $u=P(\mathcal A)$ and $v=Q(\mathcal A)$,
then the increasing event that at least $k(u+v)/2$ coordinates belong
to $\mathcal A$ would have probability difference at least
$1-2\exp(-k(u-v)^2/2)>1-2\e^{-2}>1/4$, by Hoeffding's inequality \cite{Hoeffding63}.
This contradicts the coupling.  Hence $P(\mathcal A)\le
Q(\mathcal A)+\delta$ for every increasing $\mathcal A$, and the finite
approximate form of Strassen's coupling theorem \cite{Strassen65}
yields $P\preceq_\delta Q$.  Since $k\le5\delta^{-2}$ and
$\log(2kr)\le4\log(2r/\delta)$, the displayed statement follows with
$C=4C_\ast$.  Conversely, taking $\delta=1/4$ recovers the original
conjecture after changing the constant.
\end{remark}

\paragraph{Special cases.}
We have verified \cref{conj:reserve-coupling} for rank-one laws and
independent unions of rank-one laws on disjoint blocks, including the
model above.  More generally, it holds for laws with log-concave
multiaffine generating polynomials and for strongly Rayleigh laws.
These include positively weighted basis measures of arbitrary matroids
(hence weighted spanning trees and uniform fixed-size subsets), and
finite determinantal point processes.  In these cases an increment
$C_1a$ gives failure at most $\min\{1,C_0r\exp(-c_0p/a)\}$, for absolute
positive constants $C_0,C_1,c_0$.  We omit the proofs here.

\begin{remark}\label[remark]{rem:coupling-fractional-kk}
\cref{conj:reserve-coupling} directly implies the fractional
Kahn--Kalai theorem (Theorem C).  Indeed, fractional-cover duality gives a probability law
on the minimal members of $\cF$ that is $2q_f(\cF)$-spread: the dual at
$q_f(\cF)$ has mass $1/2$ and cylinder masses at most
$q_f(\cF)^{|A|}$, and normalization gives the assertion.  Apply
\cref{conj:reserve-coupling} with $a=2q_f(\cF)$ and
$r=\ell(\cF)$.  Since $B\cup X_p\in\cF$ surely, the target product
measure belongs to $\cF$ with probability at least $3/4$. If $Ca(\log(2r)+1)\ge1$, the desired bound is trivial. Otherwise, taking
$p=Ca\log(2r)$ gives
$\pc(\cF)\le C'q_f(\cF)\log(2\ell(\cF))$.
\end{remark}

The constant increment allows the conditionally spread blocks of
\cref{sec:cores} to be absorbed successively, giving a single logarithm
in the threshold bound.

\begin{proposition}\label[proposition]{prop:coupling-implies-kkc2}
\cref{conj:reserve-coupling} implies the ``second'' Kahn--Kalai conjecture.
\end{proposition}

\begin{proof}
Take $X=E(K_n)$ and put $R=e(H)$ and $a=64\e\,q_H$.
By \cref{lem:core,lem:flag-law},
there is a chain of uniform copies
$S_0\supseteq S_1\supseteq\cdots\supseteq S_{\ell}$, with $S_0$ a copy of
$H$, such that each block $B_i:=S_i\setminus S_{i+1}$, for $i<\ell$, is conditionally
$a$-spread given $S_{i+1}$, while the last block $B_{\ell}:=S_{\ell}$ is $a$-spread.  Set
$L=\ell+1\le1+\log_{64}R$, $p_0=Ca\log(8LR)$ and
$p_i=p_0+iCa$ for $0\le i\le L$, where $C$ is the constant
in \cref{rem:coupling-amplification}.
If $p_L\geq1$, the desired threshold bound is immediate, so assume
$p_L<1$.  For $i<\ell$, condition on $S_{i+1}$ and apply
\cref{rem:coupling-amplification} with $\delta=1/(4L)$ to $B_i$
and $X_{p_i}$, chosen independently of the flag.
Its target has
the same product law $\mu_{p_{i+1}}$ for every value of $S_{i+1}$,
and is therefore independent of $S_{i+1}$.  Adding back $S_{i+1}$
and averaging gives
\[
\cL\bigl(S_i\cup X_{p_i}\bigr)
\preceq_{1/(4L)}
\cL\bigl(S_{i+1}\cup X_{p_{i+1}}\bigr),
\]
with $X_{p_j}$ independent of $S_j$ in each marginal.
Finally absorb $S_{\ell}$
in the same way.  Couplings compose with their errors adding, so
$\cL(S_0\cup X_{p_0})\preceq_{1/4}\mu_{p_L}$.  Consequently
$G_{n,p_L}$ contains $H$ with probability at least $3/4$, and
\[
\pc(H)\le p_L\le C'a\log(2R)
\le C''\pe(H)\log(2e(H)),
\]
using $L=O(\log(2R))$ and $q_H\le2\pe(H)$.
\end{proof}

\paragraph{Relation to Talagrand's convexity program.}
Talagrand's discrete convexity conjecture asks whether, for some fixed
$m$, the sets not covered by $m$ members of a family of large
$\mu_p$-measure form a $p$-small family
\cite[Research Problem 13.3.2]{Talagrand21}.
His fractional relaxation replaces smallness by the nonexistence of a
spread law supported on these uncovered sets.  His projection approach
leads to a more explicit domination question
\cite[Sections 3--4, especially Problem 4.5]{TalagrandProblems}:
if, conditional on $B$, the set $B^{(1/2)}$ is obtained by retaining
each point independently with probability $1/2$, is
$\cL(B^{(1/2)})\preceq_0\mu_{\min\{1,Ka\}}$
for every $a$-spread $B$, with an absolute $K$?
Thus spread-to-product couplings already occur explicitly in his
proposed approaches.  Our conjecture keeps all of $B$ and adds
independent noise: after a reserve of order $a\log(2r)$, it asks for
approximate domination at an additional density cost of only $O(a)$.

More recently, Li \cite[Theorem 3.3]{Li26} proved that every $a$-spread
law admits a coupling $B\subseteq R_1\cup R_2$ almost surely with
$R_1,R_2\sim\mu_a$, resolving the fractional discrete convexity
problem.
The two product samples may be dependent, so their union need not have
a product law.  This does not supply the independent reserve and the
single product target with an additive $O(a)$ increment required by
\cref{conj:reserve-coupling}.  The precise relationship with the full
discrete convexity conjecture remains to be determined.

\section{Packings and the proof of Theorem~\secref{thm:density}}\label{sec:density}

Throughout this section $H\subseteq K_n$ is nonempty,
$\Delta=\Delta(H)$ and $d_*=d_*(H)$.  By
\cref{lem:cylinder-identity}, we need to bound $\Inj(J,H)$ for every
subgraph $J\subseteq H$.  We first handle connected $J$ and then multiply
the estimates over components.

For a connected $J$ with $v$ vertices, counting along a spanning tree gives at most $n\Delta^{v-1}$ embeddings.  A maximum packing of $b$ vertex-disjoint
copies provides a second starting point: its $bv$ vertices meet every
copy of $J$, so an embedding can be rooted at one of these vertices.
This gives $bv^2\Delta^{v-1}$.  The packing bound controls $b$ in terms of
$q_H$, and taking the better of the two embedding bounds will give the
factor $\Delta^{2/d_*}$.

The packing observation (see \cref{lem:packing-bound} below) is due to Dubroff, Kahn, and Park
\cite[Claim~5.1]{DKP}, where it is used for trees.  Their companion paper
\cite{DKP2} also uses packings to organize local copy counts.  Here the
spanning-tree estimate applies to every connected witness, so the same
argument works for all $J$.

\begin{lemma}\label[lemma]{lem:embed-two-bounds}
Let $J\subseteq H$ be connected and nonempty.
Let $b_H(J)$ be the largest number of pairwise vertex-disjoint copies of $J$ in $H$.
Then
\[
\Inj(J,H)\le n\Delta^{v(J)-1}
\qquad\text{and}\qquad
\Inj(J,H)\le b_H(J)v(J)^2\Delta^{v(J)-1}.
\]
\end{lemma}

\begin{proof}
Put $v=v(J)$ and $b=b_H(J)$. Fix a spanning tree of $J$ and an ordering $x_1,\ldots,x_v$ of $V(J)$ in which
every $x_k$ with $k\geq2$ has a neighbor in the spanning tree that appears earlier in the order.  An
embedding is determined by the images in this order.  There are at most $n$
choices for the image of $x_1$, and once the image of a tree neighbor of $x_k$
is fixed, there are at most $\Delta$ choices for the image of $x_k$.  This gives
the first bound.

For the second, fix a maximum vertex-disjoint collection of $b$ copies of $J$ in
$H$ and let $X$ be the set of its $bv$ vertices.  By maximality, every copy of
$J$ in $H$ meets $X$.  Thus every embedding of $J$ maps some vertex of $J$ into
$X$.  Choose that vertex of $J$ in at most $v$ ways, choose its image in at most
$bv$ ways, and start the spanning-tree count at that vertex.
This gives at most $v\cdot bv\cdot\Delta^{v-1}$ embeddings.
\end{proof}

\begin{lemma}[{\cite[Claim~5.1]{DKP}}]\label[lemma]{lem:packing-bound}
Let $J\subseteq H$ be connected and nonempty, and let $b_H(J)$
be the maximum number of pairwise vertex-disjoint copies of $J$
in $H$. Then
\[
b_H(J)\le\frac{\e n^{v(J)}q_H^{e(J)}}{\Aut(J)}.
\]
\end{lemma}

\begin{proof}
We repeat the argument of \cite[Claim~5.1]{DKP}.
Write $v=v(J)$, $s=e(J)$ and $b=b_H(J)$.
Since the disjoint union $bJ$ is a subgraph of $H$,
the definition of $q_H$ gives $N_{bJ}(K_n)q_H^{bs}\ge1$.
As $J$ is connected, $\Aut(bJ)=\Aut(J)^b b!$. Hence
\[
1\le N_{bJ}(K_n)q_H^{bs}
=\frac{(n)_{bv}q_H^{bs}}{\Aut(J)^b b!}
\le
\left(\frac{\e n^v q_H^s}{b\,\Aut(J)}\right)^b,
\]
where we used $(n)_{bv}\le n^{bv}$ and $b!\ge(b/\e)^b$.
Taking $b$-th roots and rearranging proves the bound.
\end{proof}

We now compare the two embedding estimates.  The densest subgraph of $H$
ensures that $q_H\geq n^{-2/d_*}$; this is the scale at which the two
estimates can be combined.

\begin{lemma}\label[lemma]{lem:connected-density}
There is an absolute constant $C$ such that every connected nonempty
$J\subseteq H$ satisfies
\[
\Inj(J,H)\le n^{v(J)}\bigl(C\Delta^{2/d_*}q_H\bigr)^{e(J)}.
\]
\end{lemma}

\begin{proof}
Write $v=v(J)$, $s=e(J)$ and $x=(v-1)/s$.  Since $J$ is connected and has an
edge, $v\geq2$ and $s\geq v-1$, so $0<x\le1$.  The first bound of
\cref{lem:embed-two-bounds} gives
\[
\left(\frac{\Inj(J,H)}{n^v}\right)^{1/s}
\le
\left(\frac{\Delta^{v-1}}{n^{v-1}}\right)^{1/s}
=
\Delta^xn^{-x},
\]
while the second bound together with \cref{lem:packing-bound} gives
\[
\left(\frac{\Inj(J,H)}{n^v}\right)^{1/s}
\le
\left(\frac{\e v^2\Delta^{v-1}q_H^s}{\Aut(J)}\right)^{1/s}
\le
(\e v^2)^{1/s}\Delta^xq_H
\le
4\e\,\Delta^xq_H,
\]
as
$(\e v^2)^{1/s}\le(\e(s+1)^2)^{1/s}\le4\e$.  Combining the two,
\begin{equation}\label{eq:interpolation}
\left(\frac{\Inj(J,H)}{n^v}\right)^{1/s}
\le
4\e\,\Delta^x\min\{n^{-x},q_H\}.
\end{equation}

Let $J_*$ achieve $d_*$, and put $v_*=v(J_*)$ and $s_*=e(J_*)$. It follows from the definition
of $q_H$ that $q_H\geq N_{J_*}(K_n)^{-1/s_*}\geq n^{-v_*/s_*}=n^{-2/d_*}$.  If
$x\le2/d_*$, then $\Delta^x\le\Delta^{2/d_*}$ and we use the $q_H$ term in
\eqref{eq:interpolation}.  If $x\geq2/d_*$, then $\Delta\le n$ gives
\[
\Delta^xn^{-x}=\left(\frac\Delta n\right)^x
\le\left(\frac\Delta n\right)^{2/d_*}
\le\Delta^{2/d_*}q_H,
\]
and we use the $n^{-x}$ term.  In both cases the right-hand side of
\eqref{eq:interpolation} is at most $4\e\,\Delta^{2/d_*}q_H$.
\end{proof}

\begin{proof}[Proof of \cref{thm:density}]
Let $A\subseteq E(K_n)$ be nonempty, let $J$ be its incident graph and put
$v=v(J)$, so that $e(J)=|A|$ and $v\le2|A|$.  We may assume $\nu_H(\angles A)>0$, so $J$ is isomorphic
to a subgraph of $H$.  By \cref{lem:cylinder-identity},
\[
\nu_H(\angles A)=\frac{\Inj(J,H)}{(n)_v}.
\]
Let $J_1,\ldots,J_k$ be the components of $J$, each of which is connected,
nonempty and a subgraph of $H$.  An embedding of $J$ into $H$ gives an embedding
of each $J_i$.  These component embeddings determine the full embedding, so
$\Inj(J,H)\le\prod_{i=1}^k\Inj(J_i,H)$.  Multiplying the bounds of
\cref{lem:connected-density} over the components gives
\[
\Inj(J,H)\le n^v\bigl(C\Delta^{2/d_*}q_H\bigr)^{|A|}.
\]
As $(n)_v\geq(n/\e)^v$, we obtain
\[
\nu_H(\angles A)
\le
\e^v\bigl(C\Delta^{2/d_*}q_H\bigr)^{|A|}
\le
\bigl(C\e^2\Delta^{2/d_*}q_H\bigr)^{|A|},
\]
using $v\le2|A|$.
\end{proof}

\section{Rooted subtree sums and the proof of Theorem~\secref{thm:count}}\label{sec:trees}

Throughout this section $H\subseteq K_n$ is a tree with at least one edge. We deduce \cref{thm:count} from the following embedding bound,
whose proof occupies the remainder of the section.

\begin{theorem}\label{thm:tree-bound}
There is an absolute constant $C$ such that, for any trees
$H\subseteq K_n$ and $F$, each with at least one edge,
\[
\Inj(F,H)\le n^{v(F)}(Cq_H)^{e(F)}.
\]
\end{theorem}

\begin{proof}[Proof of \cref{thm:count} from \cref{thm:tree-bound}]
Let $A\subseteq E(K_n)$ be nonempty and let $J$ be its incident graph.
If $H$ contains no copy of $J$, then $\nu_H(\angles A)=0$.
Otherwise $J$ is a forest without isolated vertices. Restricting an
embedding to its components and applying \cref{thm:tree-bound} gives
$\Inj(J,H)\le n^{v(J)}(Cq_H)^{e(J)}$. By
\cref{lem:cylinder-identity},
\[
\nu_H(\angles A)
\le\frac{n^{v(J)}}{(n)_{v(J)}}(Cq_H)^{e(J)}
\le(C\e^2q_H)^{e(J)},
\]
where $v(J)\le2e(J)$.
\end{proof}

\subsection{Reduction to two estimates}\label{subsec:tree-reduction}

It remains to prove \cref{thm:tree-bound}. Since $v(F)=e(F)+1$,
the desired bound takes the form
$\Inj(F,H)\le n(Cnq_H)^{e(F)}$.
We establish these bounds simultaneously by showing that
\[
\sum_F \Inj(F,H)x^{e(F)}\le C_0n
\qquad\text{for }x=\frac{c}{nq_H},
\]
where $c,C_0>0$ are absolute constants and the sum runs over
isomorphism types of trees with at least one edge.
To estimate this sum, we pass to rooted subtree sums,
which allow us to exploit the recursive structure of trees.

Let $\Injr(S,T)$ count the injective homomorphisms between rooted trees
that send root to root, and put $\Autr(S):=\Injr(S,S)$.
A \emph{rooted subtree} $R\subseteq T$ is a connected subgraph containing
the root of $T$, with that same root. We also allow the one-vertex subtree.
For a rooted tree $T$, define
\begin{equation}\label{eq:coefficients}
W_T(x):=\sum_S\Injr(S,T)x^{e(S)}
       =\sum_{R\subseteq T}\Autr(R)x^{e(R)},
\end{equation}
where the first sum is over rooted isomorphism types and the second
over rooted subtrees of $T$. The equality holds because each rooted
subtree of type $S$ is the image of exactly $\Autr(S)$ root-preserving
embeddings of $S$ into $T$. 

For each tree $F$ in the sum, choose a root $s_F$, and write
$(H,o)$ for the whole tree $H$ rooted at $o$. Summing over
the possible images of $s_F$ gives, for every $x\ge0$,
\begin{equation}\label{eq:extract}
\sum_F \Inj(F,H)x^{e(F)}
=\sum_{o\in V(H)}\sum_F
  \Injr\bigl((F,s_F),(H,o)\bigr)x^{e(F)}
\le\sum_{o\in V(H)}W_{(H,o)}(x).
\end{equation}
It therefore suffices
to bound the last sum by $O(n)$ at the
chosen scale. The two propositions below do this by comparing
subtree sums with their largest terms and then averaging
those largest terms.

Define the largest term in the subtree expansion of $W_T(x)$ by
\[
\Lambda_T(x):=\max_{R\subseteq T}\Autr(R)x^{e(R)}.
\]
The singleton gives $\Lambda_T(x)\geq1$.

Fix an arbitrary reference root $r\in V(H)$. For each $u\in V(H)$, let $T_u$
be the tree spanned by $u$ and its descendants, rooted at $u$.
We retain this rooting throughout the remainder of the section.
These descendant trees differ from $(H,o)$, which denotes the whole
tree $H$ rooted at $o$.

The first estimate bounds the sum over all rootings of $H$
using descendant trees from any one fixed rooting.

\begin{proposition}\label[proposition]{prop:tree-comparison}
Let $H$ be a tree and fix $r\in V(H)$. For each $u\in V(H)$,
let $T_u$ be the tree spanned by $u$ and its descendants
in $(H,r)$, rooted at $u$. Then, for every $x\geq0$,
\[
\sum_{o\in V(H)}W_{(H,o)}(x)
\le4\sum_{u\in V(H)}\Lambda_{T_u}(1024x).
\]
\end{proposition}

The second uses the first-moment constraints defining $q_H$ to
bound the right-hand side.

\begin{proposition}\label[proposition]{prop:average}
For every fixed rooting of $H$,
\[
\sum_{u\in V(H)}\Lambda_{T_u}\!\left(\frac1{32nq_H}\right)\le2n.
\]
\end{proposition}

The proofs of \cref{prop:tree-comparison,prop:average} are given in
Sections~\ref{subsec:tree-comparison} and~\ref{subsec:tree-average},
respectively. We first show how they imply \cref{thm:tree-bound}.

\begin{proof}[Proof of \cref{thm:tree-bound}]
Take $x=1/(2^{15}nq_H)$, so that $1024x=1/(32nq_H)$.
By \eqref{eq:extract} and
\cref{prop:tree-comparison,prop:average},

\[
\begin{aligned}
\sum_F\Inj(F,H)x^{e(F)}
&\le\sum_{o\in V(H)}W_{(H,o)}(x)\\
&\le4\sum_{u\in V(H)}\Lambda_{T_u}(1024x)
\le8n.
\end{aligned}
\]
By nonnegativity, each summand is at most $8n$.
Since $e(F)\ge1$ and $v(F)=e(F)+1$, it follows that
\[
\Inj(F,H)\le8n(2^{15}nq_H)^{e(F)}
\le n^{v(F)}(2^{18}q_H)^{e(F)}. \qedhere
\]
\end{proof}

\subsection{Comparing subtree sums with their largest terms}\label{subsec:tree-comparison}

To prove \cref{prop:tree-comparison}, we first pass from whole-tree
rootings to descendant trees, then compare each resulting sum with
its largest term. For the latter step, we seek recursions that allow
the contributions from different child branches to be handled
independently. The weight $\Autr(R)$ makes this awkward: an
automorphism can permute child subtrees only when they are
rooted-isomorphic, so the factor at the root depends on the types
of the chosen child subtrees, not just on their number.

We therefore replace $\Autr(R)$ by
\[
a(R):=\prod_{w\in V(R)}d_R^+(w)!,
\]
where $d_R^+(w)$ is the number of children of $w$ in $R$.
If the root has child subtrees $R_1,\ldots,R_k$, then
$a(R)=k!\prod_{i=1}^ka(R_i)$.
This gives the simple sum and maximum recursions used below;
the degree-based formula also makes changes of root easy to control.
The following lemma shows that the total loss in replacing
$\Autr(R)$ by $a(R)$ is at most $4^{e(R)}$, which can be absorbed
by replacing $x$ with $4x$ in the subtree sums.
The comparison follows from \cite[\textup{Ch.~3}]{HP} and
\cite[\textup{Ex.~6.19}]{Stanley}.

\begin{lemma}\label[lemma]{lem:plane}
For every rooted tree $R$,
\[
\Autr(R)\le a(R)\le4^{e(R)}\Autr(R).
\]
\end{lemma}

\begin{proof}
A plane rooted tree is a rooted tree together with a linear order on the children of each vertex.
Let $p(R)$ be the number of nonisomorphic plane rooted trees
obtained by ordering the children at each vertex of $R$,
where isomorphisms preserve the root and all child orders. We claim that
$p(R)=\frac{a(R)}{\Autr(R)}$; see also \cite[Ch.~3]{HP}.

This identity implies the lemma. Indeed, if $R$ has $m$ edges,
its plane versions form a nonempty subset of all plane rooted
isomorphism types with $m$ edges. By \cite[Ex.~6.19]{Stanley},
the latter are counted by the Catalan number
$\frac1{m+1}\binom{2m}{m}$, so
$1\le\frac{a(R)}{\Autr(R)}
\le\frac1{m+1}\binom{2m}{m}\le4^m$.

It remains to establish the identity, which we prove by induction
on $e(R)$. For the one-vertex tree, all three quantities are one.
Suppose now that the root of $R$ has child subtrees
$R_1,\ldots,R_k$, with multiplicities $m_1,\ldots,m_s$ among
their rooted isomorphism types. Then
\[
a(R)=k!\prod_{i=1}^ka(R_i),
\qquad
\Autr(R)=\left(\prod_{j=1}^sm_j!\right)
               \prod_{i=1}^k\Autr(R_i).
\]
The second formula follows because an automorphism permutes
children of the same type and independently acts within each
child subtree. To form a plane version, first order the child
types in $k!/(m_1!\cdots m_s!)$ ways and then choose a plane
version for each child position. Therefore,
\[
p(R)=\frac{k!}{m_1!\cdots m_s!}\prod_{i=1}^kp(R_i).
\]
Combining these recursions gives
$\frac{p(R)\Autr(R)}{a(R)}=\prod_{i=1}^k\frac{p(R_i)\Autr(R_i)}{a(R_i)}=1$.
\end{proof}

We now approximate $W_T$ using the plane weight. Define
\[
Z_T(x):=\sum_{R\subseteq T}a(R)x^{e(R)}.
\]
By \eqref{eq:coefficients} and \cref{lem:plane},
\[
W_T(x)\le Z_T(x)\le W_T(4x).
\]

We next use the reference root $r$ to pass from whole-tree
rootings to descendant trees. Every subtree $R\subseteq H$
has a unique vertex $u$ closest to $r$ and, rooted at $u$,
is a rooted subtree of $T_u$. The following lemma bounds
the total contribution of all rootings of $R$ by its term
in $Z_{T_u}(2x)$.

\begin{lemma}\label[lemma]{lem:reroot}
For every $x\geq0$,
\[
\sum_{o\in V(H)}W_{(H,o)}(x)
\le\sum_{u\in V(H)}Z_{T_u}(2x).
\]
\end{lemma}

\begin{proof}
By \eqref{eq:coefficients}, a fixed unrooted subtree $R\subseteq H$
contributes $x^{e(R)}\sum_{o\in V(R)}\Autr(R,o)$ to the left-hand side,
where $\Autr(R,o)$ counts automorphisms fixing $o$.
Suppose $e(R)\geq1$, and let $u$ be its vertex closest to the fixed root
of $H$. Set $K(R)=\prod_{w\in V(R)}(d_R(w)-1)!$.

When $R$ is rooted at $o$, the root has $d_R(o)$ children and every other
vertex has one fewer child than its degree. Thus
$a(R,o)=d_R(o)K(R)$. By \cref{lem:plane} and the Handshaking Lemma,
\[
\sum_{o\in V(R)}\Autr(R,o)
\le\sum_{o\in V(R)}a(R,o)
=2e(R)\,K(R)\le2 e(R)\,a(R,u)\le2^{e(R)} a(R,u).
\]
Therefore all rootings of $R$ contribute at most
$a(R,u)(2x)^{e(R)}$, its term in $Z_{T_u}(2x)$. Each singleton contributes
one on both sides. Summing over the actual subtrees of $H$ proves the
inequality.
\end{proof}

It remains to compare $Z_T(x)$ with its largest term,
\[
\Gamma_T(x):=\max_{R\subseteq T}a(R)x^{e(R)}.
\]
By \cref{lem:plane}, $\Gamma_T(x)\le\Lambda_T(4x)$.
The following lemma gives the required comparison.

\begin{lemma}\label[lemma]{lem:scalar}
For every rooted tree $T$ and every $x\geq0$,
\[
Z_T(x)\le4\Gamma_T(128x).
\]
\end{lemma}

We first show how this estimate completes the proof of
\cref{prop:tree-comparison}.

\begin{proof}[Proof of \cref{prop:tree-comparison}]
By \cref{lem:reroot,lem:scalar} and the weight comparison,

\begin{align*}
\sum_{o\in V(H)}W_{(H,o)}(x)
&\le\sum_{u\in V(H)}Z_{T_u}(2x)\\
&\le4\sum_{u\in V(H)}\Gamma_{T_u}(256x)
\le4\sum_{u\in V(H)}\Lambda_{T_u}(1024x).\qedhere
\end{align*}
\end{proof}

We now prove \cref{lem:scalar}, using the recursions supplied
by the weight $a(R)$.

\begin{proof}[Proof of \cref{lem:scalar}]
Let the root of $T$ have child trees $U_1,\ldots,U_d$. Then
\begin{equation}\label{eq:recursion}
Z_T(x)=\sum_{I\subseteq[d]}|I|!\,x^{|I|}
  \prod_{i\in I}Z_{U_i}(x),\qquad
\Gamma_T(x)
=\max_{I\subseteq[d]}|I|!\,x^{|I|}
  \prod_{i\in I}\Gamma_{U_i}(x).
\end{equation}
Indeed, a rooted subtree $R\subseteq T$ is determined by a set $I$ of root children and a rooted subtree $R_i\subseteq U_i$ for every $i\in I$. For these choices, we have
$e(R)=|I|+\sum_{i\in I}e(R_i)$ and
$a(R)=|I|!\prod_{i\in I}a(R_i)$.
Summing over the choices gives the first recursive formula. Maximizing gives the second one, since the choices below different children are independent.

To close the induction with an absolute constant, we prove
the stronger bound
\[
Z_T(x)\le4\max\{1,\varepsilon\Gamma_T(128x)\},
\qquad
\varepsilon:=\exp\!\left(-\frac1{32x}\right).
\]
Since $\Gamma_T(128x)\geq1$ and
$\varepsilon\le1$, this estimate implies the lemma. The induction is on $v(T)$, and the singleton case is immediate.
For the induction step, call an index $i$ \emph{high} if
$\varepsilon\Gamma_{U_i}(128x)>1$.
Each high child supplies a factor $\varepsilon$ in the
inductive bound. With at least two high children, one such
factor absorbs the cost of summing over choices of children,
while another provides the saving required at the parent.
The cases of zero or one high child require separate estimates.

The inductive hypothesis and the first identity in
\eqref{eq:recursion} give
\[
Z_T(x)\le
\sum_{I\subseteq[d]}|I|!\,(4x)^{|I|}
\prod_{i\in I}\max\{1,\varepsilon\Gamma_{U_i}(128x)\}.
\]
The second identity in \eqref{eq:recursion}
gives, for every $I\subseteq[d]$,
\[
\Gamma_T(128x)\geq
|I|!\,(128x)^{|I|}\prod_{i\in I}\Gamma_{U_i}(128x).
\]
We will choose $I$ in this lower bound according to the number
of high indices. The factorial-sum estimates needed below are
collected in \cref{lem:sums} in Appendix~\ref{sec:scalar}.

\smallskip\noindent
\emph{No high index.}
Here every factor in the product is one, so
\[
Z_T(x)\le\sum_{j=0}^d(d)_j(4x)^j.
\]
Taking $I=[d]$ and using $\Gamma_{U_i}(128x)\geq1$ gives
$\Gamma_T(128x)\geq d!(128x)^d$.
For $dx \le1/8$, \cref{lem:sums} bounds the sum by $2$;
for $dx>1/8$, it bounds the sum by
$\varepsilon d!(128x)^d\le\varepsilon\Gamma_T(128x)$.
Thus $Z_T(x)\le\max\{2,\varepsilon\Gamma_T(128x)\}$.

\smallskip\noindent
\emph{At least two high indices.}
In this case, we have
\[
Z_T(x)\le
\left(\prod_{i\text{ high}}\varepsilon\Gamma_{U_i}(128x)\right)
\sum_{j=0}^d(d)_j(4x)^j.
\]
Taking $I=[d]$ and using $\Gamma_{U_i}(128x)\geq1$ for the other
indices gives
\[
\Gamma_T(128x)\geq
d!(128x)^d\prod_{i\text{ high}}\Gamma_{U_i}(128x).
\]
The first estimate of \cref{lem:sums} therefore gives
\[
Z_T(x)\le
\varepsilon^2\Gamma_T(128x)\exp\!\left(\frac1{128x}\right)
\le\varepsilon\Gamma_T(128x),
\]
since there are at least two high indices.

\smallskip\noindent
\emph{Exactly one high index.}
Relabel the children so that $d$ is the high index.
Splitting according to whether this child is selected gives
\[
Z_T(x)\le
\sum_{j=0}^{d-1}(d-1)_j(4x)^j
+4x\varepsilon\Gamma_{U_d}(128x)
\sum_{j=0}^{d-1}(j+1)(d-1)_j(4x)^j.
\]
If $(d-1)x \le1/8$, selecting just the high child gives
$\Gamma_T(128x)\geq128x\Gamma_{U_d}(128x)$,
and \cref{lem:sums} gives
\[
Z_T(x)\le
2+16x\varepsilon\Gamma_{U_d}(128x)
\le2+\frac{\varepsilon\Gamma_T(128x)}8.
\]
Now suppose $(d-1)x>1/8$. Selecting the other $d-1$ children, or all $d$
children, gives respectively
\[
\Gamma_T(128x)\geq(d-1)!(128x)^{d-1},
\qquad
\Gamma_T(128x)\geq d!(128x)^d\Gamma_{U_d}(128x).
\]
The remaining estimates of \cref{lem:sums} give
\[
Z_T(x) \le \varepsilon\, (d-1)!(128x)^{d-1}+\frac{1}{32}\varepsilon\, d!(128x)^d\Gamma_{U_d}(128x)\le
2\varepsilon\Gamma_T(128x).
\]
This completes the induction.
\end{proof}

\subsection{Averaging the largest terms}\label{subsec:tree-average}

Recall that $T_u$ is the subtree spanned by $u$ and its
descendants under the fixed rooting of $H$, rooted at $u$.
We prove \cref{prop:average} by choosing a maximizing rooted
subtree of each $T_u$ and grouping these subtrees by rooted
isomorphism type. If two selected subtrees intersect, one of
their roots is an ancestor of the other. The first lemma below
uses this observation to partition the occurrences of a type
with $j$ edges into at most $j+1$ families of pairwise
vertex-disjoint trees. The packing bound then controls the
size of each family.

\begin{lemma}\label[lemma]{lem:disjoint}
For each $u\in V(H)$, choose a rooted subtree $R_u\subseteq T_u$.
For every $j\ge1$, the family $\{R_u:e(R_u)=j\}$ can be
partitioned into at most $j+1$ families of pairwise
vertex-disjoint trees.
\end{lemma}

\begin{proof}
Fix $j\ge1$ and consider only the selected subtrees with
exactly $j$ edges. Order them by decreasing distance of
their roots from the fixed root of $H$, breaking ties
arbitrarily. We color them greedily using $j+1$ colors,
requiring intersecting subtrees to receive different colors.

Suppose an earlier subtree $R_v$ meets the current subtree
$R_u$ at a vertex $w$. Both $u$ and $v$ are ancestors of $w$,
so both lie on the unique path from the fixed root of $H$ to $w$. By the chosen order, $v$ must be a
proper descendant of $u$. Thus $v$ lies on the path from $u$
to $w$, which is contained in $R_u$, and hence
$v\in V(R_u)\setminus\{u\}$.

Since $R_u$ has $j$ edges and at most one selected subtree
is rooted at each vertex, at most $j$ earlier subtrees meet
$R_u$. Therefore at most $j$ colors are forbidden, and a
color is always available. The color classes give the
required partition.
\end{proof}

\begin{lemma}\label[lemma]{lem:type-count}
For each $u\in V(H)$, choose a rooted subtree $R_u\subseteq T_u$.
Let $P$ be a rooted tree with $j\geq1$ edges, and let $\bar P$
be the tree obtained by forgetting its root. The number of vertices $u$
for which $R_u$ has type $P$ is at most
\[
\frac{(j+1)\e n(nq_H)^j}{\Aut(\bar P)}.
\]
\end{lemma}

\begin{proof}
By \cref{lem:disjoint}, these subtrees split into at most $j+1$
families of pairwise vertex-disjoint copies of $\bar P$. Each family has size at most $\e n^{j+1}q_H^j/\Aut(\bar P)$ by
\cref{lem:packing-bound}. Summing gives the claim.
\end{proof}

This explains the choice of scale and weight. Multiplication by
$x^j$ at $x=1/(32nq_H)$ cancels $(nq_H)^j$. The ratio
$\Autr(P)/\Aut(\bar P)$ is at most one, so the automorphism weight
also fits within the packing bound. The remaining factor $32^{-j}$
allows us to sum over the at most $4^j$ rooted types.

\begin{proof}[Proof of \cref{prop:average}]
For each $u\in V(H)$, choose a rooted subtree $R_u\subseteq T_u$
maximizing $\Autr(R_u)/(32nq_H)^{e(R_u)}$.
Selected singletons contribute at most $v(H)\le n$.
If $k_P$ vertices select a type $P$ with $j\geq1$ edges,
then \cref{lem:type-count} gives
\[
\frac{k_P\Autr(P)}{(32nq_H)^j}
\le\e n(j+1)32^{-j},
\]
because $\Autr(P)\le\Aut(\bar P)$.
Forgetting the child orders maps plane rooted types onto rooted types,
so there are at most $4^j$ rooted types with $j$ edges.
Therefore,
\[
\sum_u\Lambda_{T_u}\!\left(\frac1{32nq_H}\right)
\le n + \e n\sum_{j\geq1}(j+1)8^{-j}
<2n.
\qedhere
\]
\end{proof}

\section*{Use of AI tools.}
The author used ChatGPT to explore ideas, assist with exposition, conduct literature searches, and check mathematical arguments. The author assumes full responsibility for the content, validity, and attribution of all mathematical claims.

\appendix

\section{Two factorial sums}\label{sec:scalar}

\begin{lemma}\label[lemma]{lem:sums}
Let $S_d(x)=\sum_{j=0}^d(d)_jx^j$ and $V_d(x)=\sum_{j=0}^d(j+1)(d)_jx^j$. The following hold.
\begin{itemize}
\item[\rm (i)] For every $x\ge 0$, we have $S_d(x)\le\exp\left(\frac{1}{32x}\right)d!(32x)^d$.
\item[\rm (ii)] For $0\le x \le 1/(2d)$, we have $S_d(x)\le2$ and $V_d(x)\le4$.
\item[\rm (iii)] For $x>1/(2d)$, we have $S_d(x)\le\exp\left(-\frac{1}{8x}\right)d!(32x)^d$ and $V_d(x)\le \exp\left(-\frac{1}{8x}\right) (d+1)!(32x)^d$.
\end{itemize}
\end{lemma}

\begin{proof}
For the first bound, substitute $k=d-j$ and use
$(d)_{d-k}=d!/k!$ to get
\[
\frac{S_d(x)}{d!(32x)^d}=\sum_{k=0}^d \frac{32^{k-d}}{k!(32x)^k} \le \sum_{k=0}^d \frac{1}{k!(32x)^k}\le \exp\left(\frac{1}{32x}\right).
\]

If $x\le 1/(2d)$, then $(d)_jx^j\le2^{-j}$.
Thus $S_d(x)\le\sum_{j\geq0}2^{-j}=2$ and
$V_d(x)\le\sum_{j\geq0}(j+1)2^{-j}=4$.

If $x>1/(2d)$, reversing the order of summation gives
\[
\frac{S_d(x)}{d!x^d}=\sum_{k=0}^d\frac{x^{-k}}{k!}
\le \exp\left(\frac{1}{x}\right).
\]
Since $x>1/(2d)$, we have
$32^{-d}\exp\left(\frac{1}{x}\right) \le \exp\left(-\frac{1}{8x}\right)$,
giving the required bound for $S_d(x)$.
Finally,
$V_d(x)\le(d+1)S_d(x) \le \exp\left(-\frac{1}{8x}\right) (d+1)!(32x)^d$.
\end{proof}

\end{document}